\documentclass[12pt,twoside,a4]{article} 
\usepackage{amsmath, amsthm, amssymb, hyperref, color, graphicx, comment, multicol}
\usepackage{amssymb}
\usepackage{graphicx}
\usepackage{amscd}
\usepackage{color}
\usepackage{float}
\usepackage{graphics}
\usepackage{amsmath}
\usepackage{amssymb}
\usepackage{mathrsfs}
\usepackage{amsthm}
\usepackage{amsfonts}
\usepackage{latexsym}
\usepackage{epsf}
\usepackage[]{hyperref}
\usepackage{fancyhdr}
\makeatletter
\newenvironment{proof*}[1][\proofname]{\par
  \pushQED{\qed}%
  \normalfont \partopsep=\z@skip \topsep=\z@skip
  \trivlist
  \item[\hskip\labelsep
        \itshape
    #1\@addpunct{.}]\ignorespaces
}{%
  \popQED\endtrivlist\@endpefalse
}
\makeatother
\date{}

\begin{document} 

%\centerline{\bf Journal's Title, Vol. x, 20xx, no. xx, xxx - xxx} 
\centerline{\bf Applied Mathematical Sciences, Vol. x, 2015, no. xx, xxx - xxx}

\centerline{\bf HIKARI Ltd, \ www.m-hikari.com}

\centerline{} 

\centerline{} 

\centerline {\Large{\bf On Generalized Fibonacci Numbers\footnote{to appear by April 2015}}} 

\centerline{} 

\centerline{\bf {Jerico B. Bacani$^*$ and Julius Fergy T. Rabago$^{\dagger}$}} 

\centerline{} 

\centerline{Department of Mathematics and Computer Science} 

\centerline{College of Science} 

\centerline{University of the Philippines Baguio} 

\centerline{Baguio City 2600, Philippines} 

\centerline{$^*$jicderivative@yahoo.com, $^{\dagger}$jfrabago@gmail.com}

\newtheorem{theorem}{Theorem}[section]
\newtheorem{lemma}[theorem]{Lemma}
\newtheorem{corollary}[theorem]{Corollary}
\newtheorem{proposition}[theorem]{Proposition}
\newtheorem{definition}[theorem]{Definition} 
\newtheorem{example}[theorem]{Example}
\newtheorem{remark}[theorem]{Remark}

\centerline{}

{\footnotesize Copyright $\copyright$ 2015 Jerico B. Bacani and Julius Fergy T. Rabago. This is an open access article distributed under the Creative Commons Attribution License, which permits unrestricted use, distribution, and reproduction in any medium, provided the original work is properly cited.}

\begin{abstract} 
We provide a formula for the $n^{th}$ term of the $k$-generalized Fibonacci-like number sequence using the $k$-generalized Fibonacci number or $k$-nacci number, and by utilizing the newly derived formula, we show that the limit of the ratio of successive terms of the sequence tends to a root of the equation $x + x^{-k} = 2$. We then extend our results to $k$-generalized Horadam ($k$GH) and $k$-generalized Horadam-like ($k$GHL) numbers. In dealing with the limit of the ratio of successive terms of $k$GH and $k$GHL, a lemma due to Z. Wu and H. Zhang \cite{wu} shall be employed. Finally, we remark that an analogue result for $k$-periodic $k$-nary Fibonacci sequence can also be derived.\\
\end{abstract} 

%{\bf Subject Classification:}  \\ 



%\end{abstract}
\section{Introduction}
A well-known recurrence sequence of order two is the widely studied \emph{Fibonacci sequence} $\{F_n\}_{n=1}^{\infty}$, which is defined recursively by the recurrence relation
	\begin{equation}\label{fibo}F_1=F_2=1,\quad F_{n+1}=F_n+F_{n-1}\quad(n\geq 1).\end{equation}
Here, it is conventional to define $F_0=0$. 

In the past decades, many authors have extensively studied the Fibonacci sequence and its various generalizations (cf. \cite{dresden, edson, lewis, koshy, noe}). We want to contribute more in this topic, so we present our results on the \emph{$k$-generalized Fibonacci numbers} or \emph{$k$-nacci numbers} and of its some generalizations. In particular, we derive a formula for the \emph{$k$-generalized Fibonacci-like sequence} using $k$-nacci numbers.
% (\emph{see} Section 2). 

Our work is motivated by the following statement:
%------------------------------- MOTIVATION --------------------------------
Consider the set of sequences satisfying the relation $\mathcal{S}_n =\mathcal{S}_{n-1} + \mathcal{S}_{n-2}$. Since the sequence $\{\mathcal{S}_n\}$ is closed under term-wise addition (resp. multiplication) by a constant, it can be viewed as a vector space. Any such sequence is uniquely determined by a choice of two elements, so the vector space is two-dimensional. If we denote such sequence as $(\mathcal{S}_0, \mathcal{S}_1)$, then the Fibonacci sequence $F_n = (0,1)$ and the shifted Fibonacci sequence $F_{n-1} = (1, 0)$ are seen to form a canonical basis for this space, yielding the identity:
	\begin{equation}\label{canonical}
		\mathcal{S}_n = \mathcal{S}_1F_n+\mathcal{S}_0F_{n-1}
	\end{equation}
for all such sequences $\{\mathcal{S}_n\}$. For example, if $\mathcal{S}$ is the Lucas sequence $2, 1, 3, 4, 7,\ldots$, then we obtain $\mathcal{S}_n\colon\!\!\!=L_n = 2F_{n-1} + F_n$. 

One of our goals in this paper is to find an analogous result of the equation \eqref{canonical} for $k$-generalized Fibonacci numbers. The result is significant because it provides an explicit formula for the $n^{th}$ term of a $k$-nacci-like (resp. $k$-generalized Horadam and $k$-generalized Horadam-like) sequences without the need of solving a system of equations.
% (note that when solving for an explicit solution of a recurrence sequence in a usual fashion the solution always proceeds in solving a system of equations caused by forcing the general solution to satisfy some prescribed initial values). 
%-------------------------------------------------------------------------------
By utilizing the formula, we also show that the limit of the ratio of successive terms of a $k$-nacci sequence tends to a root of the equation $x + x^{-k} = 2$. We then extend our results to $k$-generalized Horadam and $k$-generalized Horadam-like sequences. We also remark that an analogue result for $k$-periodic $k$-nary Fibonacci sequences can be derived. 
 %%%%%%%%%%%%%%%%%%%%%%%%
%%%%%		SECTION 1		%%%%%%%
 %%%%%%%%%%%%%%%%%%%%%%%%
 \section{Fibonacci-like sequences of higher order}
We start off this section with the following definition.
\begin{definition}
		Let $n\in\mathbb{N} \cup \{0\}$ and $k \in \mathbb{N} \backslash \{1\}$. Consider the sequences $\{F^{(k)}_n\}_{n=0}^{\infty}$ and $\{G^{(k)}_n\}_{n=0}^{\infty}$ having the following properties:
		
		\begin{equation}\label{FN}
		F^{(k)}_n = 
		\left\{
			\begin{array}{ll}
			0, & 0\leq n < k-1;\\

			1, & n=k-1;\\

			\sum_{i=1}^k F^{(k)}_{n-i}, & n > k-1,
			\end{array}
		\right.
\end{equation}
and 
	\begin{equation}\label{FLN}
		G^{(k)}_n = 
		\left\{
			\begin{array}{ll}
			G^{(k)}_n, & 0\leq n \leq k-1;\\

			\sum_{i=1}^k G^{(k)}_{n-i}, & n > k-1,
			\end{array}
		\right.
	\end{equation}
and $G^{(k)}_n \ne 0$ for some $n \in [0, k-1]$. The terms $F^{(k)}_n$ and $G^{(k)}_n$ satisfying \eqref{FN} and \eqref{FLN} are called the \emph{$n^{th}$ $k$-generalized Fibonacci} number or $n^{th}$ $k$-step Fibonacci number (cf. \cite{noe}), and \emph{$n^{th}$ $k$-generalized Fibonacci-like} number, respectively.
\end{definition}

For $\{F^{(k)}_n\}_{n=0}^{\infty}$, some famous sequences of this type are the following:
%For some values of $k$ we have the following well-known sequences and its first few non-zero terms:
	\begin{center}
	\begin{tabular}{c|c|c|}
		$k$& name of sequence & first few terms of the sequence\\ \hline \hline
		2	& Fibonacci	& $0, 1, 1, 2, 3, 5, 8, 13, 21, 34, \ldots$\\\hline
		3	& Tribonacci & $0, 0, 1, 1, 2, 4, 7, 13, 24, 44, 81 ,\ldots$\\\hline
		4	& Tetranacci & $0, 0, 0, 1, 1, 2, 4, 8, 15, 29, 56, 108,\ldots$\\\hline
		5	& Pentanacci& $0, 0, 0, 0, 1, 1, 2, 4, 8, 16, 31, 61, 120,\ldots$\\
	\end{tabular}
 	\end{center} 

By considering the sequences $\{F^{(k)}_n\}_{n=0}^{\infty}$ and $\{G^{(k)}_n\}_{n=0}^{\infty}$ we obtain the following relation.

\begin{theorem}\label{FG} Let $F^{(k)}_n$ and $G^{(k)}_n$ be the $n^{th}$ $k$-generalized Fibonacci and $k$-generalized Fibonacci-like numbers, respectively. Then, for all natural numbers $n \geq k,$
	\begin{equation}\label{GF}
		G^{(k)}_n = G^{(k)}_0 F^{(k)}_{n-1}+\sum_{m=0}^{k-3}\left(G^{(k)}_{m+1}\sum_{j=0}^{m+1} F^{(k)}_{n-1-j}\right) + G^{(k)}_{k-1}\;F^{(k)}_n.
	\end{equation}
\end{theorem}

\begin{proof}
We prove this using induction on $n$. Let $k$ be fixed. Equation \eqref{GF} is obviously valid for $n < k$. Now, suppose \eqref{GF} is true for $n \geq r \geq k$ where $r \in \mathbb{N}$.  Then, 
	\begin{align*}
	G^{(k)}_{r+1} 
		&= \sum_{i=1}^k G^{(k)}_{(r+1)-i}\\
		&= G^{(k)}_0  \sum_{i=1}^k F^{(k)}_{(r+1)-i-1}+\sum_{m=0}^{k-3}\left(G^{(k)}_{m+1}\sum_{j=0}^{m+1} \left( \sum_{i=1}^kF^{(k)}_{(r+1)-i-1-j}\right)\right)\\ 
		& \hspace{1in} + G^{(k)}_{k-1} \left( \sum_{i=1}^k F^{(k)}_{(r+1)-i}\right)\\
		&= G^{(k)}_0 F^{(k)}_{(r+1)-1}+\sum_{m=0}^{k-3}\left(G^{(k)}_{m+1}\sum_{j=0}^{m+1} F^{(k)}_{(r+1)-1-j}\right) + G^{(k)}_{k-1} F^{(k)}_{r+1}.
	\qedhere \end{align*}
%This proves the theorem.
\end{proof}
%%%%%%%%%%%%% IMPORTANT REMARK %%%%%%%%%%%%%
\begin{remark}
	Using the formula obtained by G. P. B. Dresden (cf. {\normalfont\cite[Theorem 1]{dresden}}), we can now express $G^{(k)}_n$ explicitly in terms of $n$ as follows:
		\small{
		\[
			G^{(k)}_n = G^{(k)}_0 \sum_{i=1}^kA(i;k)\alpha_i^{n-2}+\sum_{m=0}^{k-3}\left(G^{(k)}_{m+1}\sum_{j=0}^{m+1} \sum_{i=1}^kA(i;k)\alpha_i^{n-2-j}\right) + G^{(k)}_{k-1}\sum_{i=1}^kA(i;k)\alpha_i^{n-1},
		\]}
	where $A(i;k)=(\alpha_i-1)[2+(k+1)(\alpha_i-2)]^{-1}$ and $\alpha_1,\alpha_2,\ldots,\alpha_k$ are roots of $x^k-x^{k-1}-\cdots-1=0$. Another formula of Dresden for $F^{(k)}_n$ (cf. \cite[Theorem 2]{dresden}) can 
	also be used to express $G^{(k)}_n$ explicitly in terms of $n$. More precisely, we have
		\small{
		\begin{align*}
			G^{(k)}_n &= G^{(k)}_0 \sum_{i=1}^k\text{Round}\left[A(k)\alpha_i^{n-2}\right]
							+	\sum_{m=0}^{k-3}\left(G^{(k)}_{m+1}\sum_{j=0}^{m+1} \sum_{i=1}^k\text{Round}\left[A(k)\alpha_i^{n-2-j}\right]\right) \\
					&\hspace{3in}		+ G^{(k)}_{k-1}\sum_{i=1}^k\text{Round}\left[A(k)\alpha_i^{n-1}\right],
		\end{align*}
		}
		where $A(k)=(\alpha-1)[2+(k+1)(\alpha-2)]^{-1}$ for all $n\geq2-k$ and for $\alpha$ the unique positive root of $x^k-x^{k-1}-\cdots-1=0$.
\end{remark}

%%%%%%%%%%%%%%%%%%%%%%%%%%%%%%%%%%%%%%%%%%%%%%%%%%%%%%%%%%%%%%%%%%%%%%%%%%%%%%%%%%%%%%%%%%%%%%%%%%%%%%%%%%%%%%
% EXTENDING TO HORADAM NUMBERS
%%%%%%%%%%%%%%%%%%%%%%%%%%%%%%%%%%%%%%%%%%%%%%%%%%%%%%%%%%%%%%%%%%%%%%%%%%%%%%%%%%%%%%%%%%%%%%%%%%%%%%%%%%%%%%
\subsection*{Extending to Horadam numbers}
%{\bf EXTENDING TO HORADAM NUMBERS}\\
In 1965, A. F. Horadam \cite{horadam} defined a second-order linear recurrence sequence \\ 
$\{W_n(a,b;p,q)\}_{n=0}^{\infty},$ or simply $\{W_n\}_{n=0}^{\infty}$ by the recurrence relation
	\[
		W_0=a, \quad W_1=b, \quad W_{n+1} = pW_n + qW_{n-1},\quad (n\geq2).
 	\] 
The sequence generated is called the \emph{Horadam\rq{}s sequence} which can be viewed easily as a certain generalization of $\{F_n\}$. The $n^{th}$ Horadam number $W_n$ with initial conditions $W_0=0$ and $W_1=1$ can be represented by the following Binet\rq{}s formula: 
	\[
		W_n(0,1;p,q)=\frac{\alpha^n-\beta^n}{\alpha-\beta} \quad(n\geq2),
	\]
where $\alpha$ and $\beta$ are the roots of the quadratic equation $x^2-px-q=0$, i.e. $\alpha=(p+\sqrt{p^2+4q})/2$ and $\beta=(p-\sqrt{p^2+4q})$. We extend this definition to the concept of $k$-generalized Fibonacci sequence and we define the $k$-generalized Horadam (resp. Horadam-like) sequence as follows:
% , which is just an analogue of the sequence $\{F^{(k)}_n\}_{n=0}^{\infty}$ (resp. $\{G^{(k)}_n\}_{n=0}^{\infty}$) as follows: 
\begin{definition}\label{KHS}
	Let $q_i \in \mathbb{N}$ for $i \in \{1,2,\ldots,n\}$. For $n \geq k$, the $n^{th}$ $k$-generalized Horadam sequence, denoted by $\{\mathcal{U}_n^{(k)}(0,\ldots,1; q_1,\ldots,q_k)\}_{n=0}^{\infty},$ or simply $\{\mathcal{U}_n^{(k)}\}_{n=0}^{\infty},$ is a sequence whose $n^{th}$ term is obtained by the recurrence relation
	\begin{equation}\label{KH}
		\mathcal{U}^{(k)}_n 
			=  q_1\mathcal{U}^{(k)}_{n-1} + q_2\mathcal{U}^{(k)}_{n-2} + \cdots + q_k\mathcal{U}^{(k)}_{n-k}
			=\sum_{i=1}^k q_i\mathcal{U}^{(k)}_{n-i},
	\end{equation}
with initial conditions $\mathcal{U}^{(k)}_i =0$ for all $0 \leq i < k-1$ and $\mathcal{U}^{(k)}_{k-1}=1.$  Similarly, the $k$-generalized Horadam-like sequence, denoted by
$\{\mathcal{V}_n^{(k)}(a_0,\ldots,a_{k-1}; q_1,\ldots,q_k)\}_{n=0}^{\infty}$ or simply $\{\mathcal{V}_n^{(k)}\}_{n=0}^{\infty}$, has the same recurrence relation given by equation \eqref{KH} but with initial conditions $\mathcal{V}^{(k)}_i =a_i$ for all $0 \leq i \leq k-1$ where $a_{i}{}\rq{}s \in \mathbb{N} \cup \{0\}$ with at least one of them is not zero. 
\end{definition}

It is easy to see that when $q_1=\cdots=q_k=1$, then $\mathcal{U}_n^{(k)}(0,\ldots,1; 1,\ldots,1)=F^{(k)}_n$ and $\mathcal{V}_n^{(k)}(a_0,\ldots,a_{k-1}; 1,\ldots,1) = G^{(k)}_n$. Using Definition \ref{KHS} we obtain the following relation, which is an analogue of equation \eqref{GF}.

\begin{theorem}\label{UKVK}
	Let $\mathcal{U}^{(k)}_n$ and $\mathcal{V}^{(k)}_n$ be the $n^{th}$ $k$-generalized Horadam and $n^{th}$ $k$-generalized Horadam-like numbers, respectively. Then, for all $n \geq k,$
		\begin{equation}\label{VKN}
		\mathcal{V}^{(k)}_n = q_k\mathcal{V}^{(k)}_0 \mathcal{U}^{(k)}_{n-1}+\sum_{m=0}^{k-3}\left(\mathcal{V}^{(k)}_{m+1}\sum_{j=0}^{m+1} q_{k-(m+1)+j}\;\mathcal{U}^{(k)}_{n-1-j}\right) + \mathcal{V}^{(k)}_{k-1}\;\mathcal{U}^{(k)}_n.
		\end{equation}
\end{theorem}
\begin{proof}
The proof uses mathematical induction and is similar to the proof of Theorem \ref{FG}.
\end{proof}
%%%%%%%%%%%%%%%%%%%%%%%%%%%%%%%%%%%%%%%%%%%%%%%%%%%%%%%%%%%%%%%%%%%%%%%%%%%%%%%%%%%%%%%%%%%%%%%%%%%%%%%%%%%%%%
% CONVERGENCE PROPERTY
%%%%%%%%%%%%%%%%%%%%%%%%%%%%%%%%%%%%%%%%%%%%%%%%%%%%%%%%%%%%%%%%%%%%%%%%%%%%%%%%%%%%%%%%%%%%%%%%%%%%%%%%%%%%%%
\subsection*{Convergence properties}
%{\bf CONVERGENCE PROPERTIES}\\

In the succeeding discussions, we present the convergence properties of the sequences $\{F^{(k)}_n\}^{\infty}_{n=0}, \{G^{(k)}_n\}^{\infty}_{n=0},\{\mathcal{U}^{(k)}_n\}^{\infty}_{n=0}$, and $ \{\mathcal{V}^{(k)}_n\}^{\infty}_{n=0}$. First, it is known (e.g. in \cite{noe}) that $\lim_{n\rightarrow \infty} F_n^{(k)}/F_{(n-1)}^{(k)}=\alpha,$ where $\alpha$ is a \emph{$k$-nacci constant}. This constant is the unique positive real root of $x^k-x^{k-1}-\cdots-1=0$ and can also be obtained by solving the zero of the polynomial $x^k(2-x)-1$. Using this result, we obtain the following:

\begin{theorem}\label{GK}
	\begin{equation}
	\lim_{n\rightarrow \infty} G^{(k)}_n/ G^{(k)}_{n-1} = \alpha,
	\end{equation}
where $\alpha$ the unique positive root of  $x^k-x^{k-1}-\cdots-1=0.$
\end{theorem}
\begin{proof}
The proof is straightforward. Letting $n\rightarrow \infty$ in $ G^{(k)}_n/ G^{(k)}_{n-1}$ we have
	\begin{align*}
		\lim_{n\rightarrow \infty} G^{(k)}_n/ G^{(k)}_{n-1} 
		&= 
			\lim_{n\rightarrow \infty} 
			\left[\frac{G^{(k)}_0 F^{(k)}_{n-1}
			+\sum_{m=0}^{k-3}\left(G^{(k)}_{m+1}\sum_{j=0}^{m+1} F^{(k)}_{n-1-j}\right) + G^{(k)}_{k-1}\;F^{(k)}_n}{G^{(k)}_0 F^{(k)}_{n-2}
			+\sum_{m=0}^{k-3}\left(G^{(k)}_{m+1}\sum_{j=0}^{m+1} F^{(k)}_{n-2-j}\right) + G^{(k)}_{k-1}\;F^{(k)}_{n-1}}
			\right]\\
		&= 
			\lim_{n\rightarrow \infty} 
			\left[\frac{G^{(k)}_0
			+\sum_{m=0}^{k-3}\left(G^{(k)}_{m+1}\sum_{j=0}^{m+1} \frac{F^{(k)}_{n-1-j}}{F^{(k)}_{n-1}}\right) + G^{(k)}_{k-1}\;\frac{F^{(k)}_n}{F^{(k)}_{n-1}}}{G^{(k)}_0 \frac{F^{(k)}_{n-2}}{F^{(k)}_{n-1}}
			+\sum_{m=0}^{k-3}\left(G^{(k)}_{m+1}\sum_{j=0}^{m+1} \frac{F^{(k)}_{n-2-j}}{F^{(k)}_{n-1}}\right) + G^{(k)}_{k-1}}
			\right]\\
		&= 
			\frac{G^{(k)}_0
			+\sum_{m=0}^{k-3}\left(G^{(k)}_{m+1}\sum_{j=0}^{m+1} \alpha^{-j}\right) + \alpha G^{(k)}_{k-1}}{\alpha^{-1} G^{(k)}_0 
			+\sum_{m=0}^{k-3}\left(G^{(k)}_{m+1}\sum_{j=0}^{m+1} \alpha^{-(j+1)}\right) + G^{(k)}_{k-1}}\\
		&= \alpha.
\qedhere \end{align*} \end{proof}
Now, to find the limit of $\mathcal{U}^{(k)}_n/ \mathcal{U}^{(k)}_{n-1}$ (resp. $\mathcal{V}^{(k)}_n/ \mathcal{V}^{(k)}_{n-1}$) as $n \rightarrow \infty$ we need the following results due to Wu and Zhang \cite{wu}. Here, it is assumed that the $q_{i}{}\rq{}s$ satisfy the inequality $q_i \geq q_j \geq 1$ for all $j \geq i$, where $1\leq i,j \leq k$ with $2\leq k \in\mathbb{N}.$ 

\begin{lemma}\label{Q1}$\cite{wu}$
Let $q_1, q_2, \ldots , q_k$ be positive integers with $q_1 \geq q_2 \geq \cdots \geq q_k \geq 1$ and $k \in \mathbb{N} \backslash \{1\}$. Then, the polynomial 
	\begin{equation}
		f(x)= x^k - q_1x^{k-1} - q_2x^{k-2} - \cdots - q_{k-1}x-q_k,
	\end{equation}
	\begin{enumerate}
		\item[(i)] has exactly one positive real zero $\alpha$ with $q_1 < \alpha < q_1 + 1;$ and
		\item[(ii)] its other $k - 1$ zeros lie within the unit circle in the complex plane.
	\end{enumerate}
\end{lemma} 
 
\begin{lemma}\label{KE}$\cite{wu}$
Let $k\geq2$ and let $\{u_n\}_{n=0}^{\infty}$ be an integer sequence satisfying the recurrence relation given by 
	\begin{equation}
		u_n = q_1u_{n-1} + q_2u_{n-2} + \cdots + q_{k-1}u_{n-k+1} + q_ku_{n-k},\;n>k,
	\end{equation}
where $q_1,q_2,\ldots,q_k \in \mathbb{N}$ with initial conditions $u_i \in\mathbb{N}\cup\{0\}$ for $0 \leq i < k$   and at least one of them is not zero. Then, a formula for $u_n$ may be given by
 	\begin{equation}
 		u_n = c\alpha^n + \mathcal{O}(d^{-n})\;\;\; (n \rightarrow \infty),
 	\end{equation}
 where $c > 0, d > 1,$ and $q_1 < \alpha < q_1 + 1$ is the positive real zero of $f(x).$
 \end{lemma} 

We now have the following results.

\begin{theorem}
Let $\{\mathcal{U}_n\}_{n=0}^{\infty}$ be the integer sequence satisfying the recurrence relation \eqref{KH} with initial conditions $\mathcal{U}^{(k)}_i =0$ for all $0 \leq i < k-1, 2\leq k \in \mathbb{N}$ and $\mathcal{U}^{(k)}_{k-1}=1$ with $q_1 \geq q_2 \geq \cdots \geq q_k \geq 1.$ Then, 
	\begin{equation}\label{UN}
 		\mathcal{U}^{(k)}_n = c\alpha^n + \mathcal{O}(d^{-n})\;\;\; (n \rightarrow \infty),
 	\end{equation}
 where $c > 0, d > 1,$ and $\alpha \in (q_1 , q_1 + 1)$ is the positive real zero of $f(x).$ Furthermore,
 	\begin{equation}\label{UKN}
 		\lim_{n\rightarrow \infty}  \mathcal{U}^{(k)}_n/ \mathcal{U}^{(k)}_{n-1} = \alpha.
 	\end{equation}
\end{theorem}
\begin{proof}
Equation \eqref{UN} follows directly from Lemmas \ref{Q1} and \ref{KE}. To obtain \eqref{UKN}, we simply use \eqref{UN} and take the limit of the ratio $\mathcal{U}^{(k)}_n/\mathcal{U}^{(k)}_{n-1}$ as $n\rightarrow \infty$; that is, we have the following manipulation:
\begin{align*}
	\lim_{n\rightarrow \infty}  \mathcal{U}^{(k)}_n/ \mathcal{U}^{(k)}_{n-1} 
	&= \lim_{n\rightarrow \infty} \frac{ c\alpha^n + \mathcal{O}(d^{-n})}{ c\alpha^{n-1} + \mathcal{O}(d^{-(n-1)})}\\
	&=\frac{ c\alpha +  \lim_{n\rightarrow \infty} \left(\mathcal{O}(d^{-n})/\alpha^{n-1}\right)}{ c+  \lim_{n\rightarrow \infty} \left(\mathcal{O}(d^{-(n-1)})/\alpha^{n-1}\right)}\\
	&= \alpha. 
\qedhere \end{align*}
\end{proof}
Consequently, we have the following corollary.
\begin{corollary}\label{horadamlike}
Let $\{\mathcal{V}_n\}_{n=0}^{\infty}$ be an integer sequence satisfying \eqref{KH} but with initial conditions $\mathcal{V}^{(k)}_i =a_i$ for all $0 \leq i \leq k-1$ where $a_{i}{}\rq{}s \in \mathbb{N} \cup \{0\}$ with atleast one of them is not zero. Furthermore, assume that  $q_1 \geq q_2 \geq \cdots \geq q_k \geq 1,$ where $2\leq k \in \mathbb{N}$ then 
	\begin{equation}
		\lim_{n\rightarrow \infty}  \mathcal{V}^{(k)}_n/ \mathcal{V}^{(k)}_{n-1} = \alpha,
 	\end{equation} 
where $q_1 < \alpha < q_1 + 1$ is the positive real zero of $f(x).$
\end{corollary}
	\begin{proof}
		The proof uses Theorem \ref{UKVK} and the arguments used are similar to the proof of Theorem \ref{GK}.
	\end{proof}
%%%%%%%%%%%%% IMPORTANT REMARK %%%%%%%%%%%%%
\begin{remark}
	Observe that when $q_i=1$ for all $i=0,1,\ldots, k$ in Corollary \eqref{horadamlike}, then $\lim_{n\rightarrow \infty} r\colon\!\!\!= \lim_{n\rightarrow \infty} F^{(k)}_n/ F^{(k)}_{n-1} = \alpha$, where $1< \alpha <2$. Indeed, the limit of the ratio $r$ is $2$ as $n$ increases.\\
\end{remark}
%%%%%%%%%%%%%%%%%%%%%%%%%%%%%%%%%%%%%%%%%%%%%%%%%%%%%%%%%%%%%%%%%%%%%%%%%%%%%%%%%%%%%%%%%%%%%%%%%%%%%%%%%%%%%%
% K-PERIODIC FIBONACCI SEQUENCES
%%%%%%%%%%%%%%%%%%%%%%%%%%%%%%%%%%%%%%%%%%%%%%%%%%%%%%%%%%%%%%%%%%%%%%%%%%%%%%%%%%%%%%%%%%%%%%%%%%%%%%%%%%%%%%
\subsection*{$k$-Periodic Fibonacci Sequences}
%{\bf $k$-PERIODIC FIBONACCI SEQUENCES}\\	
%%%%%%%%%%%%%	COMMENT		%%%%%%%%%%%%%
%Since the Fibonacci sequence has been generalized in many ways, it is natural to ask whether we could find an analogue of \eqref{GF}, \eqref{VKN}, and \eqref{UN} for these sequences. For instance, 
In \cite{edson}, M. Edson and O. Yayenie gave a generalization of Fibonacci sequence. They called it \emph{generalized Fibonacci sequence} $\{F_n^{(a,b)}\}_{n=0}^{\infty}$ which they defined it by using a non-linear recurrence relation depending on two real parameters $(a, b)$. The sequence 
is defined recursively as
	\begin{equation}\label{edsonfib}
	F_0^{(a,b)}=0, \quad F_1^{(a,b)}=1, \quad  F_n^{(a,b)} = \left\{\begin{array}{cc}
				aF_{n-1}^{(a,b)}+F_{n-2}^{(a,b)},&\text{if}\; n\;\text{is even},\\
				bF_{n-1}^{(a,b)}+F_{n-2}^{(a,b)},&\text{if}\; n\;\text{is odd}.\\
			\end{array}\right.
	\end{equation}
This generalization has its own Binet-like formula and satisfies identities that are analogous to the identities satisfied by the classical Fibonacci sequence (\emph{see} \cite{edson}). A further generalization of this sequence, which is called \emph{$k$-periodic Fibonacci sequence} has been presented by M. Edson, S. Lewis, and O. Yayenie in \cite{lewis}. A related result concerning to two-periodic ternary sequence is presented in \cite{alp} by M. Alp, N. Irmak and L. Szalay. We expect that analogous results of \eqref{GF}, \eqref{VKN}, and \eqref{UN} can easily be found for these generalizations of Fibonacci sequence. For instance, if we alter the starting values of \eqref{edsonfib}, say we start at two numbers $A$ and $B$ and preserve the recurrence relation in \eqref{edsonfib}, then we obtain a sequence that we may call \emph{$2$-periodic Fibonacci-like sequence}, which is defined as follows:
	\begin{equation}\label{edsonfiblike}
	G_0^{(a,b)}=A, \quad G_1^{(a,b)}=B, \quad  G_n^{(a,b)} = \left\{\begin{array}{cc}
				aG_{n-1}^{(a,b)}+G_{n-2}^{(a,b)},&\text{if}\; n\;\text{is even},\\
				bG_{n-1}^{(a,b)}+G_{n-2}^{(a,b)},&\text{if}\; n\;\text{is odd}.\\
			\end{array}\right.
	\end{equation} 
The first few terms of $\{F_n^{(a,b)}\}_{n=0}^{\infty}$ and $\{G_n^{(a,b)}\}_{n=0}^{\infty}$ are as follows:

\[
\begin{array}{l|c|c}
n	&F_n^{(a,b)}		&G_n^{(a,b)}\\\hline\hline
0	&0  				& A \\
1	&1  				& B \\
2	&a  				& aB+A  \\
3	&ab+1  			& (ab+1)B + bA \\
4	&a^2b+2a 			& (a^2b+2a)B + (ab+1)A\\
5	&a^2b^2+3ab+1 		& (a^2b^2+3ab+1)B + (ab^2+2b)A\\
6	&a^3b^2+4 a^2 b +3 a   & (a^3b^2+4 a^2 b +3 a)B + (a^2b^2+3ab+1)A\\
7	&a^3 b^3+5 a^2 b^2+ 6 a b+ 1	& (a^3 b^3+5 a^2 b^2+ 6 a b+ 1)B+ (a^2 b^3 + 4 a b^2 + 3 b) A\\
%\vdots& \vdots & \vdots
\end{array}
\]
Suprisingly, by looking at the table above, $G_n^{(a,b)}$ can be obtained using $F_n^{(a,b)}$ and $F_n^{(b,a)}$. 
%(notice that we have interchanged the parameter $a$ and $b$ in the latter expression).
 More precisely, we have the following result.
\begin{theorem}\label{abG} Let $F_n^{(a,b)}$ and $G_n^{(a,b)}$ be the $n^{th}$ terms of the sequences defined in \eqref{edsonfib} and \eqref{edsonfiblike}, respectively. Then, for all $n \in \mathbb{N}$, the following formula holds
	\begin{equation}\label{abGform}
		G^{(a,b)}_n = G_1^{(a,b)} F_n^{(a,b)} + G_0^{(a,b)} F_{n-1}^{(b,a)}.
	\end{equation}
\end{theorem}

\begin{proof}
	The proof is by induction on $n$. Evidently, the formula holds for $n=0,1,2$. We suppose that the formula also holds for some $n \geq 2$. Hence, we have
		\begin{align*}
			G^{(a,b)}_{n-1} 	&= G_1^{(a,b)} F_{n-1}^{(a,b)} + G_0^{(a,b)} F_{n-2}^{(b,a)},\\
			G^{(a,b)}_n 	&= G_1^{(a,b)} F_n^{(a,b)} + G_0^{(a,b)} F_{n-1}^{(b,a)}.
		\end{align*}
	Suppose that $n$ is even. (The case when $n$ is odd can be proven similarly.) So we have
		\begin{align*}
			G^{(a,b)}_{n+1} 	&= aG^{(a,b)}_n + G^{(a,b)}_{n-1}\\
						&= a \left( G_1^{(a,b)} F_n^{(a,b)} + G_0^{(a,b)} F_{n-1}^{(b,a)}\right) + \left( G_1^{(a,b)} F_{n-1}^{(a,b)} + G_0^{(a,b)} F_{n-2}^{(b,a)}\right)\\
						&= G_1^{(a,b)} \left( a F_n^{(a,b)} +  F_{n-1}^{(a,b)} \right) +G_0^{(a,b)} \left( a F_{n-1}^{(b,a)} + F_{n-2}^{(b,a)}\right)\\
						&= G_1^{(a,b)} F_{n+1}^{(a,b)} + G_0^{(a,b)} F_n^{(b,a)},
		\end{align*}
	proving the theorem.
\end{proof}

%%%%%%%%%%%%%%%%%%%%%%%%%%%%%%%%%%%%%%%%%%%%%%%%%%%%%%%%%%%%%%%%%%%%%%%%%%%%%%%%%%%%%%%%%%%%%%%
	The sequence $\{G_n^{(a,b)}\}_{n=0}^{\infty}$ has already been studied in (\cite{edson}, Section 4). The authors \cite{edson} have related the two sequences $\{F_n^{(a,b)}\}_{n=0}^{\infty}$ and $\{G_n^{(a,b)}\}_{n=0}^{\infty}$ using the formula
	\begin{equation}\label{abGform2}
		G_n^{(a,b)}= G_1^{(a,b)}F_n^{(a,b)}+G_0^{(a,b)}\left(\frac{b}{a}\right)^{n-2\lfloor n/2\rfloor}F_{n-1}^{(a,b)}.
	 \end{equation}
Notice that by simply comparing the two identities \eqref{abGform} and \eqref{abGform2}, we see that
	\[
		F_{n-1}^{(b,a)}=\left(\frac{b}{a}\right)^{n-2\lfloor n/2\rfloor}F_{n-1}^{(a,b)}, \quad \forall n \in \mathbb{N}.
	\]
The convergence property of $\{F_{n+1}^{(a,b)}/F_n^{(a,b)}\}_{n=0}^{\infty}$ has also been discussed in (\cite{edson}, Remark 2). It was shown that, for $a=b$,  we have
	\begin{equation}\label{conv}
		\frac{F_{n+1}^{(a,b)}}{F_n^{(a,b)}} \longrightarrow \frac{\alpha}{a}=\frac{a+\sqrt{a^2+4}}{2}\quad \text{as}\quad n\longrightarrow \infty.  
	\end{equation}
Using \eqref{abGform} and \eqref{conv}, we can also determine the limit of the sequence $\{G_{n+1}^{(a,b)}/G_n^{(a,b)}\}$ as $n$ tends to infinity, and for $a=b$, as follows:
	\begin{align*}
		\lim_{n\rightarrow \infty}\frac{G_{n+1}^{(a,b)}}{G_n^{(a,b)}}		
			&= \lim_{n\rightarrow \infty}\frac{G_1^{(a,b)} F_{n+1}^{(a,b)} + G_0^{(a,b)} F_n^{(b,a)}}{G_1^{(a,b)} F_n^{(a,b)} + G_0^{(a,b)} F_{n-1}^{(b,a)}}
			= \lim_{n\rightarrow \infty}\frac{G_1^{(a,a)} F_{n+1}^{(a,a)} + G_0^{(a,a)} F_n^{(a,a)}}{G_1^{(a,a)} F_n^{(a,a)} + G_0^{(a,a)} F_{n-1}^{(a,a)}}\\
			&= \frac{G_1^{(a,a)}\lim_{n\rightarrow \infty}\frac{ F_{n+1}^{(a,a)}}{F_n^{(a,a)}} + G_0^{(a,a)} }{G_1^{(a,a)}  + G_0^{(a,a)} \lim_{n\rightarrow \infty}\frac{F_{n-1}^{(a,a)}}{F_n^{(a,a)}}}
			=\frac{\alpha a^{-1}G_1^{(a,a)} + G_0^{(a,a)} }{G_1^{(a,a)} + a \alpha^{-1} G_0^{(a,a)} }=\displaystyle\frac{\alpha}{a}.
	\end{align*}
For the case $a \neq b$, the ratio of successive terms of $\{F_n^{(a,b)}\}$ does not converge. However, it is easy to see that
	\[
		\frac{F_{2n}^{(a,b)}}{F_{2n-1}^{(a,b)}} \longrightarrow \frac{\alpha}{b},\quad \frac{F_{2n+1}^{(a,b)}}{F_{2n}^{(a,b)}} \longrightarrow \frac{\alpha}{a}, \quad \text{and}
			\quad \frac{F_{n+2}^{(a,b)}}{F_n^{(a,b)}} \longrightarrow \alpha + 1,
	\]
where $\alpha=(ab+\sqrt{a^2b^2+4ab})/2$ (cf. \cite{edson}). Knowing all these limits, we can investigate the convergence property of the sequences $\{G_{2n}^{(a,b)}/G_{2n-1}^{(a,b)}\}, \;\{G_{2n+1}^{(a,b)}/G_{2n}^{(a,b)}\}$, and $\{G_{n+2}^{(a,b)}/G_n^{(a,b)}\}$. Notice that $F_n^{(a,b)}=F_n^{(b,a)}$ for every $n \in \{1,3,5, \ldots\}$. So 
\begin{comment}
	\begin{align*}
		\lim_{n\rightarrow \infty}\frac{G_{2n+1}^{(a,b)}}{G_{2n}^{(a,b)}}		
			&= \lim_{n\rightarrow \infty}\frac{G_1^{(a,b)} F_{2n+1}^{(a,b)} + G_0^{(a,b)} F_{2n}^{(b,a)}}{G_1^{(a,b)} F_{2n}^{(a,b)} + G_0^{(a,b)} F_{2n-1}^{(b,a)}}\\
			&= \lim_{n\rightarrow \infty} \frac{F_{2n+1}^{(a,b)}}{F_{2n-1}^{(b,a)}} \frac{G_1^{(a,b)} + G_0^{(a,b)} \frac{F_{2n}^{(b,a)}}{F_{2n+1}^{(a,b)}}}{G_1^{(a,b)} \frac{F_{2n}^{(a,b)}}{F_{2n-1}^{(b,a)}} + G_0^{(a,b)}}\\
			&= \lim_{n\rightarrow \infty} \frac{F_{2n+1}^{(a,b)}}{F_{2n-1}^{(a,b)}} \frac{G_1^{(a,b)} + G_0^{(a,b)} \frac{F_{2n}^{(b,a)}}{F_{2n+1}^{(b,a)}}}{G_1^{(a,b)} \frac{F_{2n}^{(a,b)}}{F_{2n-1}^{(a,b)}} + G_0^{(a,b)}}\\
			&= (\alpha+1) \left( \frac{G_1^{(a,b)} + b\alpha^{-1}G_0^{(a,b)} }{\alpha b^{-1}G_1^{(a,b)}+ G_0^{(a,b)}}\right)
			= \frac{\alpha(\alpha+1)}{b}.
	\end{align*}
\end{comment}	
	\begin{align*}
		\lim_{n\rightarrow \infty}\frac{G_{2n}^{(a,b)}}{G_{2n-1}^{(a,b)}}		
			&= \lim_{n\rightarrow \infty}\frac{G_1^{(a,b)} F_{2n}^{(a,b)} + G_0^{(a,b)} F_{2n-1}^{(b,a)}}{G_1^{(a,b)} F_{2n-1}^{(a,b)} + G_0^{(a,b)} F_{2n-2}^{(b,a)}}
			= \lim_{n\rightarrow \infty} \frac{G_1^{(a,b)} \frac{F_{2n}^{(b,a)}}{F_{2n-1}^{(a,b)}}+ G_0^{(a,b)}}{G_1^{(a,b)}+ G_0^{(a,b)} \frac{F_{2n-2}^{(a,b)}}{F_{2n-1}^{(b,a)}}}\\
			&= \lim_{n\rightarrow \infty} \frac{G_1^{(a,b)} \frac{F_{2n}^{(b,a)}}{F_{2n-1}^{(b,a)}}+ G_0^{(a,b)}}{G_1^{(a,b)}+ G_0^{(a,b)} \frac{F_{2n-2}^{(a,b)}}{F_{2n-1}^{(a,b)}}}
			= \frac{ \alpha a^{-1}G_1^{(a,b)}+ G_0^{(a,b)}}{G_1^{(a,b)}+ a\alpha^{-1}G_0^{(a,b)}}
			=\frac{\alpha}{a}.
	\end{align*}
Similarly, it can be shown that $G_{2n+1}^{(a,b)}/G_{2n}^{(a,b)} \rightarrow \alpha/b$ and $G_{n+2}^{(a,b)}/G_n^{(a,b)} \rightarrow \alpha+1$ as $n\rightarrow \infty$.  

The recurrence relations discussed above can easily be extended into subscripts with real numbers. For instance, consider the piecewise defined function $G_{\lfloor x\rfloor}^{(a,b)}$: 
	\begin{equation}\label{gofx}
	G_0^{(a,b)}=A, \; G_1^{(a,b)}=B, \;  G_{\lfloor x\rfloor}^{(a,b)} = \left\{\begin{array}{cc}
				aG_{{\lfloor x\rfloor}-1}^{(a,b)}+G_{{\lfloor x\rfloor}-2}^{(a,b)},&\text{if}\; {\lfloor x\rfloor}\;\text{is even},\\
					&\\
				bG_{{\lfloor x\rfloor}-1}^{(a,b)}+G_{{\lfloor x\rfloor}-2}^{(a,b)},&\text{if}\; {\lfloor x\rfloor}\;\text{is odd}.\\
			\end{array}\right.
	\end{equation} 

Obviously, the properties of \eqref{edsonfiblike} will be inherited by \eqref{gofx}. For example, suppose $G_0^{(a,b)}=2, \; G_1^{(a,b)}=3, \; a=0.2$, and $b=0.3$. Then, $G^{(.2,.3)}_{\lfloor x\rfloor} = G_1^{(.2,.3)} F_{\lfloor x\rfloor}^{(.2,.3)} + G_0^{(.2,.3)} F_{{\lfloor x\rfloor}-1}^{(.3,.2)}$. Also, 
	\[	
		\lim_{x\rightarrow \infty}\frac{G_{2{\lfloor x\rfloor}}^{(.2,.3)}}{G_{2{\lfloor x\rfloor}-1}^{(.2,.3)}} 	= 1.3839, \quad
		\lim_{x\rightarrow \infty}\frac{G_{2{\lfloor x\rfloor}+1}^{(.2,.3)}}{G_{2{\lfloor x\rfloor}}^{(.2,.3)}} 	= 0.921886,\quad
		\lim_{x\rightarrow \infty}\frac{G_{{\lfloor x\rfloor}+2}^{(.2,.3)}}{G_{{\lfloor x\rfloor}}^{(.2,.3)}}		= 1.276807. 
	\]
If $a=b=0.1$, then the ratio of successive terms of $\{G_n^{(.1,.1)}\}$ with $G_0^{(.1,.1)}=2$ and $G_1^{(.1,.1)}=3$ converges to $1.05125$. See Figure \ref{fig} for the plots of these limits.
%-------------------------------------------- 	EXAMPLE --------------------------------------------
\begin{figure}[h!]
\begin{multicols}{2} 
    \includegraphics[width=.45\textwidth]{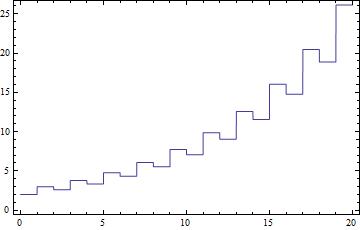}
    \includegraphics[width=.45\textwidth]{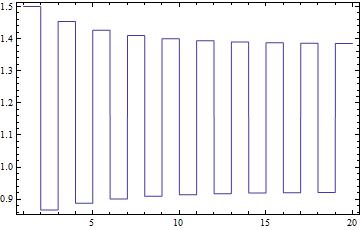}
    \includegraphics[width=.45\textwidth]{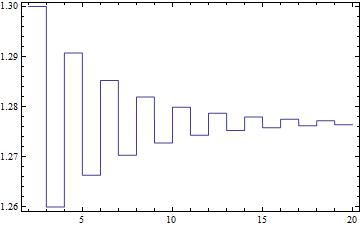}
    \includegraphics[width=.45\textwidth]{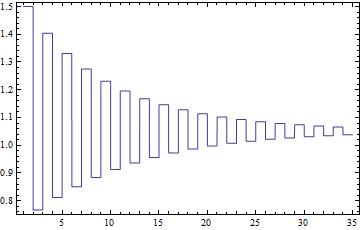}    
    \end{multicols}
	\caption{$G_0 = 2, G_1=3, a=0.2, b=0.3;\; G_0 = 2, G_1=3, a=0.1, b=0.1$.}\label{fig}
\end{figure}

%------------------------------------------ FURTHER GENERALIZTIONS AND OBSERVATIONS ------------------------------------------------
Now, we may take the generalized Fibonacci sequence \eqref{edsonfib} a bit further by considering a 3-periodic ternary recurrence sequence related to the usual Tribonacci sequence:
	\begin{align}
	&\quad\quad T_0^{(a,b,c)}=0, \quad T_1^{(a,b,c)}=0, \quad T_2^{(a,b,c)}=1, \nonumber\\ 
		 T_n^{(a,b,c)} = 
		 	&\left\{\begin{array}{cc}
				a T_{n-1}^{(a,b,c)} + T_{n-2}^{(a,b,c)} + T_{n-3}^{(a,b,c)},&\text{if}\; n \equiv 0\; (\text{mod}\; 3),\\
				b T_{n-1}^{(a,b,c)} + T_{n-2}^{(a,b,c)} + T_{n-3}^{(a,b,c)},&\text{if}\; n \equiv 1\; (\text{mod}\; 3),\\
				c T_{n-1}^{(a,b,c)} + T_{n-2}^{(a,b,c)} + T_{n-3}^{(a,b,c)},&\text{if}\; n \equiv 2\; (\text{mod}\; 3).\\
			\end{array}\right. (n\geq 3) \label{tribo}
	\end{align}
Suppose we define the sequence $\{U_n^{(a,b,c)}\}_{n=0}^{\infty}$ satisfying the same recurrence equation as in \eqref{tribo} but with arbitrary initial conditions $U_0^{(a,b,c)}, \; U_1^{(a,b,c)}, \;$ and $U_2^{(a,b,c)}$. Then, the sequences $\{U_n^{(a,b,c)}\}_{n=0}^{\infty}$ and $\{T_n^{(a,b,c)}\}_{n=0}^{\infty}$ are related as follows:
 	\[
 		U_n^{(a,b,c)} = U_0^{(a,b,c)} T_{n-1}^{(b,c,a)}+ U_1^{(a,b,c)} \left( T_{n-1}^{(b,c,a)}+T_{n-2}^{(c,a,b)}\right) + U_2^{(a,b,c)} T_n^{(a,b,c)}, \quad (n \geq 2).
 	\]
 	%----------------------------------- REMARK ON GENERALIZATION --------------------------------------------
\begin{remark}
	In general, the $k$-periodic $k$-nary sequence $\{\mathcal{F}_n^{(a_1,a_2,\ldots,a_k)}\}\colon\!\!\!=\{\mathfrak{F}_n^{(k)}\}$ related to $k$-nacci sequence:
	\begin{align}
	&\mathfrak{F}_0^{(k)}=\mathfrak{F}_1^{(k)}=\cdots=\mathfrak{F}_{k-2}^{(k)}=0, \quad \mathfrak{F}_{k-1}^{(k)}=1, \nonumber\\ 
		 &\mathfrak{F}_n^{(k)} = 
		 	\left\{\begin{array}{cc}
				a_1 \mathfrak{F}_{n-1}^{(k)} + \sum_{j=2}^k \mathfrak{F}_{n-j}^{(k)},&\text{if}\; n \equiv 0\; (\text{mod}\; k),\\
				a_2 \mathfrak{F}_{n-1}^{(k)} + \sum_{j=2}^k \mathfrak{F}_{n-j}^{(k)},&\text{if}\; n \equiv 1\; (\text{mod}\; k),\\
					&\vdots\\
				a_k \mathfrak{F}_{n-1}^{(k)} + \sum_{j=2}^k \mathfrak{F}_{n-j}^{(k)},&\text{if}\; n \equiv -1\; (\text{mod}\; k).\\
			\end{array}\right. (n\geq k) \label{kbo}
	\end{align}
	and the sequence $\{\mathcal{G}_n^{(a_1,a_2,\ldots,a_k)}\}\colon\!\!\!=\{\mathfrak{G}_n^{(k)}\}$ defined in the same recurrence equation \eqref{kbo} but with arbitrary initial conditions $\mathfrak{G}_0^{(k)}, \mathfrak{G}_1^{(k)}, \ldots, \mathfrak{G}_{k-1}^{(k)}$ are related in the following fashion:
	\begin{equation}\label{genGF}
		\mathfrak{G}^{(k)}_n = \mathfrak{G}^{(k)}_0 \mathcal{F}^{(k;1)}_{n-1}+\sum_{m=0}^{k-3}\left(\mathfrak{G}^{(k)}_{m+1}\sum_{j=0}^{m+1} \mathcal{F}^{(k;j+1)}_{n-1-j}\right) + \mathfrak{G}^{(k)}_{k-1}\;\mathfrak{F}^{(k)}_n,
	\end{equation}
where $(k;j)\colon\!\!\!=(a_{j+1},a_{j+2},\ldots,a_k, a_1,a_2,\ldots,a_j), j=0,1,\ldots,k-1$. This result can be proven by mathematical induction and we leave this to the interested reader.
\end{remark}

%%%%%%%%%% REFERENCES %%%%%%%%%%%%

{\bf Received: March 3, 2015}

\end{document}